\documentclass[12pt]{article}

\usepackage{amsfonts,amssymb,latexsym,xspace,epsfig,graphics,color}
\usepackage{amsmath,amsthm,enumerate,stmaryrd,xy}
\usepackage{bbm}
\usepackage{subfigure}
\usepackage{harvard}
\citationstyle{dcu}

\newcommand{\bX}{\mathbf{X}}
\newcommand{\bY}{\mathbf{Y}}
\newcommand{\Z}{\mathbb{Z}}
\newcommand{\N}{\mathbb{N}}
\newcommand{\X}{\mathbf X}
\newcommand{\Y}{\mathbf Y}

\def\square{\ifmmode\sqr\else{$\sqr$}\fi}
\def\one{{\bf 1}\hskip-.5mm}

\def\P{{\mathbb P}}

\def\Z{{\mathbb Z}}

\def\N{{\mathbb N}}

\def\Y{{\bf Y}}

\def\b{{\bf b}}

\def\sqr{\vcenter{
         \hrule height.1mm
         \hbox{\vrule width.1mm height2.2mm\kern2.18mm\vrule width.1mm}
         \hrule height.1mm}}                  
\def\square{\ifmmode\sqr\else{$\sqr$}\fi}
\def\one{{\bf 1}\hskip-.5mm}

\newtheorem{lemma}{Lemma}[section]

\newtheorem{defi}{Definition}[section]
\newtheorem{theo}{Theorem}[section]
\newtheorem*{theo*}{Theorem}
\newtheorem{prop}{Proposition}[section]

\newtheorem{notalert}{Notation alert}[section]
\newtheorem{example}{Example}[section]

\begin{document}


\title{Continuity properties of a factor of Markov chains}

\author{Walter A. F. de Carvalho$^{1}$, Sandro Gallo$^{2}$, and Nancy L. Garcia$^{1}$\thanks{Corresponding author; $^{1}$ University of Campinas; $^{2}$ Federal University of Rio de Janeiro; Brasil}} 

\maketitle

\begin{abstract}
Starting from a Markov chain with a finite alphabet, we consider the
chain obtained when all but one symbol are undistinguishable
for the practitioner. We study necessary and sufficient conditions for
this  chain to {have continuous transition probabilities with respect 
to the past.} 
\end{abstract}

\section{Introduction}

  Consider a positive recurrent, stationary Markov chain $\bX = \{X_t,
  t \in \Z\}$ in a finite (or countable) alphabet $A$ with transition
  matrix $P^{\bf X}$ and a factor map, $\pi: A \rightarrow B$ with
  $\textrm{card}(B)<\textrm{card}(A)$. Hidden Markov models obtained
  through factor mappings of Markov Chains have received a lot of
  attention for a long time and { an exhaustive listing would be
    impossible. We therefore mention only some paper that are most
    directly related to our interest. \citeasnoun{harris:1955} proved
    that, in general, the image process $\bY= \{Y_t=\pi(X_t), t \in \Z\}$ is a
    chain with infinite order {with continuous transition
      probabilities decaying exponentially fast if $P$ is strictly
      positive (i.e., $P^{\bf X}_{ij} >0$ for all $i,j \in A$). This result
      was then extended in \citeasnoun{verb:2011} and
      \citeasnoun{verb:2011b}, to the case where the original chain
      ${\bf X}$ is not necessarily a Markov chain, but only needs to
      be a chain with strictly positive and continuous transition
      probabilities.  In the statistical physics and dynamical systems
      contexts, \citeasnoun{chazottes/ugalde/2003} were interested in
      determining whether the image process of a Markov process
      satisfies (Bowen's) Gibbsianess.  Finally, another related study
      is that of preservation of (order 1) Markovianess, we refer for instance
      to \citeasnoun{burk:rose:1958} who gave necessary and sufficient
      conditions under which $\bY$ is still a Markov process.}

In this work, we focus our efforts on the special case where $B=
\{0,1\}$ and $\pi^{-1}(1)$ is a singleton. Without loss of generality,
we take $A = \{1, \ldots, m\}$ and $\pi(x) = \one(x=1)$. In this case,
it is well-known that the image process {$\bY= \{Y_t, t \in \Z\}$ is
  always a discrete time binary renewal chain, and it is an easy
  matter to obtain an explicit matricial form for its transition
  probabilities. This explicitness then allows us to search not only
  for sufficient conditions (as in most of the previously cited
  article), but also for necessary conditions under which the image
  process has continuous transition probabilities. We refer the reader
  to Section \ref{sec:cont} for a more detailed explanation of our
  objectives, assumptions and results.  }

The paper is organized as follows. In Section \ref{sec:notation} we
  introduce the notation and basic definitions. The main results and
  their proofs are given in Section \ref{sec:cont}, each theorem being
  directly followed by its proof. We conclude with a section
  containing some examples and a discussion comparing our results to
  those of the related literature.

\section{Notation and {basic definitions}}\label{sec:notation}

\noindent{\bf Notation.}  

\vspace{0.3cm}

\noindent Let $A$ be a countable alphabet. Given two
integers $m\leq n$, we denote by $a_m^n$ the string $a_m \ldots a_n$
of symbols in $A$.
 For any $m\leq n$, the length of the string $a_m^n$
is denoted by $|a_m^n|$ and defined by $n-m+1$.
  Given two strings $v$
and $v'$, we denote by $vv'$ the string of length $|v| + |v'| $
obtained by concatenating the two strings. The concatenation of
strings is also extended to the case where $v=\ldots a_{-2}a_{-1}$ is
a semi-infinite sequence of symbols. Let $A^{-\mathbb{N}}=A^{\{\ldots,-2,-1\}}$
be the set of all infinite strings of past symbols. Finally, we denote by $\underline{a}=\ldots a_{-2}a_{-1}$ the elements of $A^{-\mathbb{N}}$.
%

\vspace{0.3cm}

\noindent{\bf Kernel, compatibility and continuity.}
\begin{defi}[Kernel] A \emph{transition probability kernel} (or simply \emph{kernel} in the sequel) on a countable alphabet $A$ is a (measurable) function
\begin{equation}
\begin{array}{cccc}
P:&A\times A^{-\mathbb{N}}&\rightarrow&
[0,1]\\ &(a,\underline{a})&\mapsto&P(a|\underline{a})
\end{array}
\end{equation}
such that $\sum_{a\in A}P(a|\underline{a})=1$ for any $\underline{a}\in A^{-\mathbb{N}}$.
\end{defi}

For instance, a given kernel $P$ is Markovian of order $k\ge1$ if 
$P(a|\underline{a})=P(a|\underline{b})$ for any $\underline{a}$ and $\underline{b}$ such that
$a_{-k}^{-1}=b_{-k}^{-1}$. So the above definition is a natural extension of the notion of transition matrix, the later being a terminology commonly used for $k$-steps Markov chains. 

A  kernel is sometimes called set of transition probabilities, a terminology that we used in the introduction. It is important to notice that these are transition probabilities with respect to the ``past''. When considering transition probabilities with respect to both, past and future, as we will do in Section \ref{sec:discussion} when comparing our results to the notion of Gibbsianess, we will use the terminology \emph{specification}.

\vspace{0.3cm}

In this paper, stochastic chains are regarded as bi-infinite sequence of random variables whose dependence information is given by the kernel. We therefore need to define what we mean by a stochastic chain being compatible with a given kernel.

\begin{defi}[Compatibility] \label{def:chain}A stationary stochastic chain ${\bf X}=(X_{n})_{n\in\Z}$ is said to be \emph{compatible} with a kernel $P$ if the later is a regular version of the conditional probabilities of the former, that is
\begin{equation}\label{compa}
\P(X_{0}=a_0|X_{-\infty}^{-1}=a_{-\infty}^{-1})=P(a|a_{-\infty}^{-1})
\end{equation}
for every $a_0\in A$ and $\P$-a.e. $a_{-\infty}^{-1}$ in $A^{-\mathbb{N}}$.
\end{defi}

Let us emphasize that stationarity here means that for any $n\ge1$ and any $a_{0}^{n-1}$, $\P(X_{0}^{n-1}=a_{0}^{n-1})=\P(X_{k}^{k+n-1}=a_{0}^{n-1})$ for any $k\in \Z$.

As an example for Definition \ref{def:chain}, a stationary $1$-step Markov chain satisfies that $\P(X_0=a_0|X_{-1}=a_{-1})=P_{a_{-1}a_0}$ for any $a_{-1},a_0\in A$, where $P_{a_{-1}a_0}:=P(a_0|\underline{a})$, since this later does not depend on $a_{-\infty}^{-2}$. In what follows, $P_{ij},\,i,j\in A$, will always denote the entries of a transition matrix on $A$.

\vspace{0.3cm}

%

We finally need to introduce what we mean by continuous kernel, the main notion of interest of the present paper.

\begin{defi}[Continuity]\label{def:cont} A  kernel $P$ is
\emph{continuous} (with respect to the product topology) at some point
$\underline{a}$ if $P(a|a_{-i}^{-1}\underline{z})\rightarrow P(a|\underline{a})$
whenever $i$ diverges, for any $\underline{z}$. Naturally, a kernel $P$ is \emph{continuous} if it is continuous at every point. 
\end{defi}
In the above definition, $P(a|a_{-i}^{-1}\underline{z})$ stands for the transition probability from the (concatenated) past $\underline{z}a_{-i}^{-1}$ to the symbol $a$, we reversed time in the conditioning.

It is clear that $k$-steps Markov kernels are continuous since for them, $P(a|a_{-i}^{-1}\underline{z})=P(a|a_{-k}^{-1}\underline{z})$ for any $i\ge k$. So continuous kernels constitute a natural
extension of $k$-steps Markov kernels. The next definition aims to quantify \emph{how much} continuous is a given kernel.

\begin{defi}[Continuity rate] \label{txcont}
The continuity rate of a  kernel $P$ is defined, for any $k\ge1$, by
$$\beta(k)\,:=\, \sup_{a_{-k}^{0}}\sup_{\underline{b},\underline{c}}|P(a_0|a_{-k}^{-1}\underline{b})-P(a_0|a_{-k}^{-1}\underline{c})|.$$
\end{defi}

Observe that if the alphabet $A$ is finite, the compactness of $A^{-\N}$ implies that $P$ is continuous if and only if its continuity rate converges to $0$.

\section{Results and Proofs}\label{sec:cont}

One of the main assumption of the present paper is that we restrict our considerations to a particular case of mapping. 

\begin{defi}[Aggregation map]\label{def:ag} A map $\pi: A \rightarrow \{0,1\}$ defined by $\pi(a) = \one(a=1)$ will be called {\rm aggregation map}. \end{defi}

\begin{notalert} We will use the same notation $\pi$ to denote the map $\pi:A^n \rightarrow \{0,1\}^n$ defined by $\pi(a_{1}^{n})_i = \one(a_i=1)$. The generalization for infinite sequences and processes is analogous. \end{notalert}

In the sequel, we will essentially study two stochastic chains, the original Markov chain ${\bX}$ and its image $\bY$. in order to avoid confusions, we will  index quantities by the related stochastic chains: $P^{{\bf X}}$ and $P^\bY$ for the kernels, and  $\beta^{{\bf X}}(k)$ and $\beta^\bY$ for the continuity rates.

\subsection{Preliminary considerations}
%
Focussing on aggregation maps (see Definition \ref{def:ag}) allows to obtain a  closed formula for the continuity rate of the image process, as stated by the next proposition. 

 For any $\underline{a}\in\{0,1\}^{-\N}$, let
\[
\ell(\underline{a}):=\inf\{i\ge0:a_{-i-1}=1\},
\]
using the convention that $\ell(\ldots00)=\infty$. Also let us introduce the $[0,1]$-valued sequence $p_k,\,k\ge0$, defined by $p_0:=P_{11}=\P(X_0=1|X_{-1}=1)$ and for any $k\ge1$,
\begin{equation}\label{eq:q}
p_k:=\frac{\sum_{b_{-k}^{-1}\in\pi^{-1}(0_{-k}^{-1})}\P(X_{-k}^{-1}=b_{-k}^{-1},\,X_{0}=1|X_{-k-1}=1)}{\sum_{b_{-k}^{-1}\in\pi^{-1}(0_{-k}^{-1})}\P(X_{-k}^{-1}=b_{-k}^{-1}|X_{-k-1}=1)}.
\end{equation}

\begin{prop}\label{prop:preliminary}
If ${\bf X}$ is a stationary positive recurrent Markov chain, then the continuity rate of $P^{{\bf Y}}$ is
\[
\beta^{{\bf Y}}(k)\le \sup_{l,m\ge k}|p_l-p_m|.
\]
\end{prop}

\begin{proof}
For all $k\ge0$, we will compute, for any $i\ge1$ and any $a_{-i}^{-1}\in \{0,1\}^{i}$, 
\begin{equation}\label{eq:1}
\P(Y_{0}=1|Y_{-i-k-1}^{-k-2}=a_{-i}^{-1},\,Y_{-k-1}=1,\,Y_{-k}^{-1}=0_{-k}^{-1}).
\end{equation}
It equals, by the definition of conditioning
\begin{align}\label{eq:2}
\frac{\P(Y_{-i-k-1}^{-k-2}=a_{-i}^{-1},\,Y_{-k-1}=1,\,Y_{-k}^{-1}=0_{-k}^{-1},\,Y_{0}=1)}{\P(Y_{-i-k-1}^{-k-2}=a_{-i}^{-1},\,Y_{-k-1}=1,\,Y_{-k}^{-1}=0_{-k}^{-1})}\end{align}
Now, using that $\pi$ is the aggregating function of Definition \ref{def:ag} and that ${\bf X}$ is Markovian, we have that the numerator of \eqref{eq:2} equals
\[
\sum_{c_{-i}^{-1}\in\pi^{-1}(a_{-i}^{-1})}\P(X_{-i-k-1}^{-k-2}=c_{-i}^{-1},\,X_{-k-1}=1)\sum_{b_{-k}^{-1}\in\pi^{-1}(0_{-k}^{-1})}\P(X_{-k}^{-1}=b_{-k}^{-1},\,X_{0}=1|X_{-k-1}=1),
\]
and analogously, the denominator equals
\[
\sum_{c_{-i}^{-1}\in\pi^{-1}(a_{-i}^{-1})}\P(X_{-i-k-1}^{-k-2}=c_{-i}^{-1},\,X_{-k-1}=1)\sum_{b_{-k}^{-1}\in\pi^{-1}(0_{-k}^{-1})}\P(X_{-k}^{-1}=b_{-k}^{-1}|X_{-k-1}=1).
\]
Thus, the fraction in \eqref{eq:2} factorizes and we  obtain that \eqref{eq:1} equals $p_k$,
independently of $a_{-k}^{-1}$.  
This in particular means that, for any $\underline{a}\in\{0,1\}^{-\N}\setminus\{\ldots00\}$, $\P(Y_0=1|Y_{-k}^{-1}=a_{-k}^{-1})$ converges, since it is constant for any $k\ge \ell(\underline{a})+1$. By \citeasnoun{kalikow/1990}, ${\bf Y}$ has a continuous kernel $P^{{\bf Y}}$ if and only if $\P(Y_0=1|Y_{-k}^{-1}=a_{-k}^{-1})$ converges uniformly on $\{0,1\}^{-\N}$, which in the present case, due to compactness of $\{0,1\}^{-\N}$, amounts to say that $P^\bY$ is continuous if and only if $\P(Y_0=1|Y_{-k}^{-1}=a_{-k}^{-1})$ converges for any $\underline{a}\in\{0,1\}^{-\N}$. Since this quantity converges for any $\underline{a}\in\{0,1\}^{-\N}\setminus\{\ldots00\}$, we have continuity if and only if the convergence occurs at $\ldots00$. 

Now, by the definition of compatibility, we can take $P^{\bY}$ defined for any $\underline{a}$ by $P^{{\bf Y}}(1|\underline{a})=\lim_k\P(Y_0=1|Y_{-k}^{-1}=a_{-k}^{-1})$. In other words, these limits can be taken to be the regular version of the conditional probabilities. In particular, $P^\bY(1|\underline{a})=p_{\ell(\underline{a})}$ for any $\underline{a}\in\{0,1\}^{-\N}\setminus\{\ldots00\}$ and we can take $P^\bY(1|\ldots00)=\lim_kp_k=:p_\infty$ to preserve continuity.

Then, we can compute the continuity rate of $P^\bY$:
\begin{align*}
\beta^\bY(k)&:= \sup_{a_{-k}^{-1}}\sup_{\underline{b},\underline{c}}|P^\bY(1|a_{-k}^{-1}\underline{b})-P^\bY(1|a_{-k}^{-1}\underline{c})|\\
&=\sup_{\underline{b},\underline{c}}|P^\bY(1|0_{-k}^{-1}\underline{b})-P^\bY(1|0_{-k}^{-1}\underline{c})|\\
&=\sup_{\underline{b},\underline{c}}|p_{\ell(\underline{b})+k}-p_{\ell(\underline{c})+k}|\\
&=\sup_{l,m\ge k}|p_l-p_m|
\end{align*}
where in the second line we used the fact that if $a_{-i}=1$ for some $i\in\{1,\ldots,k\}$, then as we said above $P^\bY(1|a_{-k}^{-1}\underline{b})=P^\bY(1|a_{-k}^{-1}\underline{c})=p_i$, and in the third line, we used the fact that the transition probability $P^\bY(1|\underline{a})$ only depends on $\ell(\underline{a})$.

%

\end{proof}

In the above argument, observe that the value of $P^\bY(1|\ldots00)$ can be chosen arbitrarily.
For instance, we could have taken a different value than the limit of $p_k$, creating a discontinuity
at $\ldots00$ according to Definition \ref{def:cont}. But this is only an artificial discontinuity, because  the regular version $P^\bY$ is defined up to sets of null measure.  Thus we will always choose the continuous one. 

In certain cases however, it is impossible to make $P^\bY$ continuous  just changing its value on a set of null measure. For instance, let
$\bX$ be a Markov chain with alphabet $A = \{1,2,3\}$ and transition
matrix given by

$$P^{{\bf X}} = \left( \begin{array}{ccc}
                    \alpha & 1 - \alpha & 0 \\
                    \beta &    0        & 1 - \beta \\
                    \gamma & 1 - \gamma & 0 
                    \end{array}
                    \right).$$
Let $\pi: A \rightarrow \{0,1\}$ be such that $\pi(1) =
1$, $\pi(2)=\pi(3)=0$. In this case, simple calculations show that $p_0=\alpha$, for $k\ge1$ odd, $p_k = \beta$ and for $k\ge2$ even, $p_k =\gamma$. {What happens then is, borrowing the terminology used in Gibbs/Non-Gibbs literature that the kernel $P^{{\bf Y}}$ has an \emph{essential discontinuity} at the point $\ldots00$, independently of the value  $P(1|\ldots00)$ we choose. We refer to \citeasnoun{fernandez/gallo/maillard/2011} for a complete discussion.}

\vspace{0.3cm}

The main objective of this paper is to understand which condition for
$P^{\bf X}$ leads to a continuous/discontinuous process ${\bf Y}$.  To
do so, Proposition \ref{prop:preliminary} says that it suffices to
obtain sufficient conditions for $p_k$ to converge (or not). To study
$p_k$, we are going to use a decomposition of the transition matrix
$P^{\X}$ similar to the one used by by \citeasnoun{darr:sene:1967} in
their study of quasi-stationary measures. This decomposition will
allow us to use Perron-Forbenius theory.

\vspace{0.3cm}

Recall that $P^\bX$ is an $m\times m$ stochastic matrix. Let 
\begin{equation} \label{Matricial} 
  P^{\X}=\left(\begin{array}{cccc}
P_{11} & V \\
W^t & P
\end{array} \right)
  \end{equation}
where 
 \begin{itemize}
 \item $V$ is the row vector $(P_{12}\ldots P_{1m})$; 
 \item $W$ is the row vector $(P_{21}\ldots P_{m1})$, and $W^t$ denotes the transpose of $W$;  
 \item $P$ is the $(m-1)\times(m-1)$ sub-matrix $P_{ij}, i,j\in \{2,\ldots,m\}$. 
 \end{itemize} 

Denote by ${\bf 1}$ and ${\bf 0}$ the row vectors formed by ones and
zeros, respectively. To avoid trivial cases, we will always consider
$V \neq {\bf 0}$ and $W \neq {\bf 0}$. 

\vspace{0.3cm}

\begin{notalert}In the sequel, $P$ will always denote the sub-matrix of $P^\bX$. The kernels will always be indexed by the corresponding process.
\end{notalert}

Recall that $p_0=P_{11}$. We can now rewrite $p_k$, $k\ge1$ (see \eqref{eq:q}) in the following form
\begin{equation} \label{transagreg} 
p_{k+1}=\frac{VP^kW^t}{VP^k{\bf 1}^t}, \,\,\,\,\mbox{ if } k\geq0 
\end{equation}
with $P^0=I$ being the identity $(m-1)\times(m-1)$ matrix.

Before we come to the statements of the results, let us recall the following very basic definitions concerning matrices. 

\begin{defi} A square matrix $(P_{ij})$ is called {\rm  irreducible} if, for any $i,j$ there exists an $n(i,j)$ such that $P^{n}_{ij}>0$. Otherwise, it is called {\rm reducible}. \end{defi}

\begin{defi} A square matrix $(P_{ij})$ is called {\rm primitive} if it is non-negative and there is an $n > 0$ such that $P^{n}_{ij} \neq 0$ for all $i,j$.
\end{defi}

\begin{defi} Let $(P_{ij})$ be non-negative. Fix an state $i$ and define the {\rm period of state $i$} as
$$ h(i)\, := \, {\rm gcm} \{n: P^{n}_{ii} > 0\}.$$
In the case of irreducible matrices, every state has the same period $h$ and this number is called the {\rm period of $P$}. If $h=1$, $P$ is called {\rm aperiodic}. 
\end{defi}
Thus, primitive matrices are irreducible aperiodic
non-negative matrices.

\subsection{The irreducible case}

Let ${\bf X}$ be a stationary Markov process with
space state $A=\{1,2,\ldots,m\}$ and its transition matrix $P^{\bf X}$ decomposed as
(\ref{Matricial}) with $P$ being irreducible matrix. Let $\Y=\pi(\X)$ where $\pi$ is an aggregation
map. Theorems \ref{thm:1} will give necessary and
sufficient conditions under which the process $\Y$ is continuous as well
as its continuity rate.
 \begin{theo} \label{thm:1} 
\begin{enumerate}
\item If $P$ is primitive, then  $P^{\Y}$ is continuous. Moreover
  $$\beta^{\Y}(n) \,=\,
O\left(n^{d_2-1}\left(\frac{|\lambda_2|}{\lambda_1}\right)^n\right)$$
where $\lambda_1$ is the maximal eigenvalue of $P$, $\lambda_2$ and
$d_2$ are, respectively, the value and multiplicity of the second
largest eigenvalue in absolute value.
\item If $P$ has period $h$, then the process $\Y$ is continuous if, and only if, 
\begin{equation} \label{Condcicl}
 \frac{V.G^*.P^r.W^t}{V.G^*.P^r.{\bf 1}^t}\,=\,\frac{V.G^*.W^t}{V.G^*.{\bf 1}^t} 
 \end{equation}
for all $r=1, \ldots, h-1$, where $G^*=\phi^*.(\psi^{*})^t>0$, and $\phi^*$ and $\psi^*$ are the Perron vectors corresponding to the Perron value $\lambda_1^*$ from $P^h$ and $(P^{h})^t$, respectively. 
\end{enumerate}
\end{theo}

\begin{proof} {\bf If $P$ is primitive}, then there exits an unique maximal
eigenvalue $\lambda_1$ (Perron-Frobenius eigenvalue or Perron root),
which is strictly greater in absolute value than all other
eigenvalues and $\lambda_1$  is the spectral radius of $P$.
Therefore,
$$\lim_{n\rightarrow\infty}\left(\frac{P}{\lambda_1}\right)^n=G$$
and 
  \begin{equation} \label{Teoprim}
  p_\infty=\lim_{n\rightarrow\infty}p_n=\frac{VGW^t}{VG{\bf 1}^t}
  \end{equation} 
where $G=\phi.\psi^t>0$, and $\phi$ and $\psi$ are the Perron vectors corresponding to $\lambda_1$ from $P$ and $P^t$, respectively. 

Moreover, let $\lambda_2$ be the second largest eigenvalue in absolute
value and $d_2$ its multiplicity. Then, it is known (Theorem 1.2,
p. 9, \citeasnoun{sene:2006}) that
 \begin{equation} \label{conv.lambda}
\left(\frac{P}{\lambda_1}\right)^n=G+O\left(n^{d_2-1}.\left(\frac{|\lambda_2|}{\lambda_1}\right)^n\right).
 \end{equation}

\begin{eqnarray*}
\beta^{\Y}(n)= \left|p_n-p_\infty\right| 
=\left|\frac{V.P^n.W^t}{V.P^n.{\bf 1}^t} \,-\, \frac{V.G.W^t}{V.G.{\bf 1}^t}\right| 
= O\left(n^{d_2-1}.\left(\frac{|\lambda_2|}{\lambda_1}\right)^n\right).
\end{eqnarray*}

\vspace{.2cm}

{\bf If $P$ is not primitive}, without loss of generality we can relabel
the symbols on $A$ such that $P$ can be written in the Frobenius form
decomposed into primitive blocks as follows
  \begin{equation} \label{P1cicl}
P=\left(\begin{array}{ccccccccccccccccccccccccc}
{\bf 0} & B_{1} & {\bf 0} & \ldots & {\bf 0} \\
{\bf 0} & {\bf 0} & B_{2} & \ldots & {\bf 0} \\
\vdots & \vdots & \vdots & \ddots & \vdots \\
{\bf 0} & {\bf 0} & {\bf 0} & \ldots & B_{h-1} \\
B_{h} & {\bf 0} & {\bf 0} & \ldots & {\bf 0} 
\end{array} \right).
 \end{equation}
Therefore, $P^h$ will be a primitive matrix given by
 \begin{equation} \label{Phcicl}
P^h=\left(\begin{array}{ccccccccccccccccccccccccc}
B_{1}.B_{2}.\ldots.B_{h} & {\bf 0} & \ldots & {\bf 0} \\
{\bf 0} & B_{2}.B_{3}.\ldots.B_{h}.B_{1} & \ldots & {\bf 0} \\
\vdots & \vdots & \ddots & \vdots \\
{\bf 0} & {\bf 0} & \ldots & B_{h}.B_{1}.\ldots.B_{h-1} 
\end{array} \right)
\end{equation}
and we have 
$$\lim_{n\rightarrow\infty}\left(\frac{P^h}{\lambda_1^*}\right)^n=G^*$$
with $G^*=\phi.\psi^t>0$, and $\phi$ and $\psi$ are the right and left
Perron vectors corresponding to $\lambda_1$ from $P^h$,
respectively. Since,
$$p_{n.h+1}=\frac{V.P^{n.h}.W^t}{V.P^{n.h}.{\bf 1}^t}=\frac{V.\left(\frac{P^h}{\lambda_1^*}\right)^{n}.W^t}{V.\left(\frac{P^h}{\lambda_1^*}\right)^{n}.{\bf 1}^t}$$
and
$$p_{n.h+r+1}=\frac{V.P^{n.h+r}.W^t}{V.P^{n.h+r}.{\bf 1}^t}=\frac{V.\left(\frac{P^h}{\lambda_1^*}\right)^n.P^r.W^t}{V.\left(\frac{P^h}{\lambda_1^*}\right)^n.P^r.{\bf 1}^t},$$
leading to
\begin{equation} \label{eq:cr}
\lim_{n\rightarrow\infty}p_{n.h+r+1}=\lim_{n\rightarrow\infty}\frac{V.\left(\frac{P^h}{\lambda_1^*}\right)^{n}.P^r.W^t}{V.\left(\frac{P^h}{\lambda_1^*}\right)^{n}.P^r.{\bf 1}^t}=\frac{V.G^*.P^r.W^t}{V.G^*.P^r.{\bf 1}^t}.
\end{equation}
Therefore, in order to have continuity of the transtion probabilites, we need Expression (\ref{eq:cr}) to be constant for all $r=0, \ldots, h-1$. 
\end{proof}

In some cases, the aggregated process has finite memory, for example, it is
known that if the sum of all rows of $P$ are equal then the resulting
process is Markovian \cite{burk:rose:1958}. It is possible to obtain higher
order Markov processes as shown by the following proposition.

\begin{prop} \label{prop:markov}
 (a) The process $\Y$ is a Markov process with order at most $m-1$ if $P$
 have a single non-null eigenvalue. \\
(b) The process $\Y$ is a Markov process with order $k+1$ if $P$ have a single non-null eigenvalue and $k$ non-null rows in its triangular form.
\end{prop}

\begin{proof}
Under the conditions of the theorem, then $\lambda_1$ is the only
non-null eigenvalue and by Theorem 1.2, p.9 (Seneta, 2006) we have
 $$P^n=\lambda_1^{n}.\phi.\psi^t$$
for all $n\geq m-2$. Therefore, for any $n \ge m-2$ 
$$p_{n+1}=\frac{V.P^n.W^t}{V.P^n.{\bf 1}^t}=\frac{V.\lambda_1^n.\phi.\psi^t.W^t}{V.\lambda_1^n.\phi.\psi^t.{\bf 1}^t}=\frac{V.G.W^t}{V.G.{\bf 1}^t}.$$
    
Moreover, according to Schur’s Triangularization Theorem, if all eigenvlaues
of $P$ are real then $P$ can be triangularized, that is, there exist
$Q$ and $\Delta$, such that
  $$P=Q.\Delta.Q^t \quad \mbox{ and } \quad P^n=Q.\Delta^n.Q^t $$
where $\Delta$ is triangular superior and $Q$ is orthogonal. Also,
$\Delta$ is composed by the eigenvalues in its diagonal. Suppose that
$\Delta$ has $k$ non-null rows. Then, it is easy to see that for $n
\ge k$
  $$\Delta^n=\lambda_1^{n-k}.\Delta^k.$$
Now, for $n \ge k$
\begin{eqnarray*}
p_{n+1}&=&  \frac{V.Q.\Delta^{n}.Q^t.W^t}{V.Q.\Delta^{n}.Q^t.{\bf 1}^t} \\
&=& \frac{\lambda_1^{n-k}.V.Q.\Delta^{k}.Q^t.W^t}{\lambda_1^{n-k}.V.Q.\Delta^{k}.Q^t.{\bf 1}^t} \\
&=&  \frac{V.Q.\Delta^{k}.Q^t.W^t}{V.Q.\Delta^{k}.Q^t.{\bf 1}^t} \\ 
&=& p_{k+1}.
\end{eqnarray*}
Therefore the process $\Y$ is an order $k+1$ Markov chain.
\end{proof}

Notice that this Proposition does not cover the usual order 1 Markov
case since $\Delta$ cannot have zero non-null rows (the second largest
eigenvalue is equal to zero, therefore at least one row of $\Delta$ is
equal to zero).

\subsection{The reducible case}

Without loss of generality consider that the transition submatriz $P$ can be written as 
\begin{equation} \label{canored}
   P=\left(\begin{array}{ccccccccccccccccccccccccc}
T_{11} & T_{12} & T_{13} & \ldots & T_{1c} \\
{\bf 0} & T_{22} & {\bf 0} & \ldots & {\bf 0} \\
{\bf 0} & {\bf 0} & T_{33} & \ldots & {\bf 0} \\
\vdots & \vdots & \vdots & \ddots & \vdots \\
{\bf 0} & {\bf 0} & {\bf 0} & \ldots & T_{cc} 
\end{array} \right).
\end{equation}
where the submatrices $T_{22},T_{33},\ldots,T_{cc}$ are irreducible.

The reducible case is harder to analyse in generality since the
matrices $T_{22},T_{33},\ldots,T_{cc}$ can be periodic with different
periods. However, this situation can be studied case by case using the same
tools as described before. 

There are some special cases, where we can get sufficient conditions
for contuinuity. 

Let $\lambda_{1,i}$ be the Perron value of $T_{ii}$ and define
\begin{equation} 
\label{eq:l1}
\lambda_{\max}=\max_{i} \lambda_{1,i}.
\end{equation}

\begin{defi} \label{Bdominante}
 The matrices $T_{ii},i\in\{2,3,...,c\}$ with $\lambda_{i} = \lambda_{\max}$
are called  {\rm Dominating Blocks}.
\end{defi}

Let
\begin{equation}
\label{eq:Delta}
\Delta = \{i\in\{2, \ldots, c\}: T_{ii} \mbox{ is a Dominating block}\}.
\end{equation}

\begin{theo}  \label{thm:2a}
If the matrix $P$ is reducible and written as (\ref{canored}),  then
the process $\Y$ is continuous if all Dominating Blocks ($T_{ii}$ for $i \in \Delta$) are
\begin{enumerate}[i]
\item either irreducible; or
\item are periodic with the same period $h$ and for all  $i \in \Delta$
 \begin{equation} \label{Condciclred}
 \frac{V_{i} G_{i} (T_{ii})^r W_{i}^t}{V_i G_i (T_{ii})^r {\bf 1}_{i}^t}=\frac{V_{i} G_{i} W_{i}^t}{V_i G_i {\bf 1}_{i}^t},
 \end{equation}
for all $r = 0, \ldots, h-1$ with $G_{i}= \phi_i \psi_i^t>0$, and $\phi_i$ and $\psi_i$ are the right and left Perron vectors corresponding to $\lambda_{1,i}$ from $T_{ii}^h$.
\end{enumerate}
\end{theo}

\begin{proof}
Analogous to Theorem \ref{thm:1}. 
\end{proof}

\subsection{Denumerable alphabet}

The theory of Perron-Frobenius can be used for countable, real,
non-negative matrices. The definitions of irreducibility and period
extend easily to such matrices. Moreover, we can define the
Perron-value of a denumerable matrix using the following lemma.

\begin{lemma}[Lemma 7.1.1, \citeasnoun{kitch:1998}] Let $P$ be a countable, real, non-negative, irreducible and aperiodic matrix. For a fixed state $i$:
\\ (i) there exists a $k$ such that $(P^n)_{ii} > 0$ for all $n \ge
k$; \\ (ii) $(P^{n+m})_{ii} \ge (P^n)_{ii} \cdot (P^m)_{ii}$;\\ (iii)
$\lim_{n \rightarrow \infty} \sqrt{n}{(P^n)_{ii}}$ exists and equals
$\sup_n \sqrt{n}{(P^n)_{ii}}$; \\ (iv) if $\lambda = \lim_{n
  \rightarrow \infty} \sqrt{n}{(P^n)_{ii}}$ then $(P^n)_{ii}/\lambda^n
\le 1$ for all $n$.
\end{lemma}

Then, $\lambda$ is the {\it Perron value} of $P$. It can be
infinite. For this section, we will assume it is finite. However, this
is not enough to have the desired results in terms of continuity. We
will need positive recurrence. This concept is well understood for
stochastic matrices in terms of the time to return to the
states. However, for general non-negative matrices the following
definitions are necessary.  Let
$$t_{ij}(0) = \delta_{ij}, \quad t_{ij}(1) = P_{ij} \quad \mbox{ and }\quad
t_{ij}(n) = (P^n)_{ij}.$$
Define the generating function as
\begin{equation}
\label{eq:gf1}
T_{ij}(z) = \sum_{n=0}^{\infty} t_{ij}(n) z^n.
\end{equation}
Now, let
$$\ell_{ij}(0) = 0, \quad \ell_{ij}(1) = P_{ij} \quad \mbox{ and }
\ell_{ij}(n+1) = \sum_{r \neq i} \ell_{ir}(n) t_{rj}.$$
%
Observe that the radius of convergence of $T_{ij}(z)$ is $1/\lambda$
for all $i,j$. 

\begin{defi}
The irreducible matrix $P$ corresponding to the generating function (\ref{eq:gf1})
is: \\
(i) {\rm recurrent} if $T_{ii}(1/\lambda) = \infty$; \\
(ii) {\rm transient} if $T_{ii}(1/\lambda) < \infty$; \\
(iii) {\rm positive recurrent} if $T_{ii}(1/\lambda) = \infty$ and
$\sum_{n} n \ell_{ii}(n)/\lambda^n < \infty$; \\
(iv) {\rm null recurrent} if $T_{ii}(1/\lambda) = \infty$ and
$\sum_{n} n \ell_{ii}(n)/\lambda^n = \infty$. 
\end{defi}
Notice that the definition is independent of the choice of $i$. 

The Generalized Perron-Frobenius Theorem states the following. 

\begin{theo*}[Theorem 7.1.3, \citeasnoun{kitch:1998}]
Suppose $P$ is a countable, non-negative, irreducible, aperiodic and positive
recurrent matrix. Then, there exists a Perron value $\lambda > 0$
(assumed to be finite) such that $\lim_{n \rightarrow \infty}
P^n/\lambda^n = \phi.\psi^t>0$, and $\phi$ and $\psi$ are the Perron
vectors corresponding to $\lambda$ from $P$ and $P^t$, respectively.
\end{theo*}

Therefore, we have the following result for countable Markov chains.
\begin{theo} \label{thm:1a}
If the matrix $P$ is countable, non-negative, irreducible, aperiodic
and positive recurrent matrix, then the process $\Y$ is continuous.
\end{theo}

\section{Discussion}\label{sec:discussion}

\subsection{Examples}

\begin{example}{\bf All entries positive -- comparing with Harris' rate of convergence.}
Consider the Markov chain ${\bf X}$ with alphabet $A = \{1,2,3\}$ and
transition matrix given by
$$P^{\bf X} = \left(\begin{array}{lll}
                  0.10 &  0.3 & 0.60 \\
                  0.20 &  0.3 & 0.50 \\
                  0.05 &  0.7 & 0.25 
                  \end{array}
             \right).
$$
\end{example}
Notice that the process ${\bf Y}$ is not Markov since $W$ is not
constant. In this case, it is easy to compute the eigenvalues and
eigenvectors of $P$ and $P'$ corresponding to the largest eigenvalue
obtaining that $p_\infty=\lim_{n\rightarrow\infty}p_n=0.132864$ and
the rate of convergence 
 $$\beta^{\Y}(n) \,=\,
O\left(\left(n \frac{|\lambda_2|}{\lambda_1}\right)^n\right)= O\left(n \, 
  0.3657281^n\right).$$

Notice that all entries of $P^{\bf X}$ are positive and the continuity
rate given in \citeasnoun{harris:1955} (Eq. (6.4)) gives a bound for
continuity rate as
$$\beta^{\Y}(n) \le (1 - \lambda)^{n-1}$$
where
$$ \lambda = \min_{i,j,k,l} \frac{p_{kj}p_{il}}{|A|^2 p_{ij} p_{kl}}.$$
In this case, $\lambda = 0.01190476$ and 
$$\beta^{\Y}(n) \le (0.9880952)^{n-1}.$$

\begin{example}{\bf $P$ is reducible and ${\bf Y}$ is continuous.}
Consider the Markov chain ${\bf X}$ with  alphabet $A = \{1,2,\ldots,7\}$ and transition matrix given by 
$$P^{\bf X} = \left(\begin{array}{ccccccc}
                  0.2 &  0.3 & 0.1 & 0.2 & 0.1 & 0   & 0.1 \\
                  0.2 &  0.3 & 0.2 & 0.1 & 0   & 0.1 & 0.1 \\
                  0.1 &  0.1 & 0.2 & 0   & 0.3 & 0.1 & 0.2 \\
                  0   &  0   & 0   & 0.4 & 0.6 & 0   & 0   \\
                  0.3 &  0   & 0   & 0.2 & 0.5 & 0   & 0   \\
                  0.7 &  0   & 0   & 0   & 0   & 0.3 & 0   \\
                  0.6 &  0   & 0   & 0   & 0   & 0   & 0.4
                  \end{array}
             \right).
$$
\end{example}
Since we have $P$ to be reducible we have to find the dominant blocks
and the maximum eigenvalue of the dominant blocks. In this case, the
dominant block is 
$$T_{22} = \left(\begin{array}{ll}
                   0.4 & 0.6    \\
                   0.2 & 0.5   
                  \end{array}
             \right).
$$
which has eigenvalues $0.8$ and $0.1$ and left and right normalized
eigenvectors $\psi = (0.5150787, 1.0301574)'$ and
$\phi=(0.8320503,0.5547002)'$ respectively yielding 
$p_\infty=\lim_{n\rightarrow\infty}p_n=0.2$ and the rate of convergence is
 $$\beta^{\Y}(n) \,=\,
O\left(\left(n \frac{|\lambda_2|}{\lambda_1}\right)^n\right)= O\left(n
  (1/8)^n\right).$$

\begin{example}{\bf $P$ is periodic with constant sum on its columns.}
Consider the Markov chain ${\bf X}$ with  alphabet $A =
\{1,2,\ldots,m\}$ and transition matrix such that
  \begin{equation} \label{P1cicl}
P=\left(\begin{array}{ccccc}
{\bf 0} & B_{1} & {\bf 0} & \ldots & {\bf 0} \\
{\bf 0} & {\bf 0} & B_{2} & \ldots & {\bf 0} \\
\vdots & \vdots & \vdots & \ddots & \vdots \\
{\bf 0} & {\bf 0} & {\bf 0} & \ldots & B_{h-1} \\
B_{h} & {\bf 0} & {\bf 0} & \ldots & {\bf 0} 
\end{array} \right)
 \end{equation}
with $\sum_i P_{ij} = k$ for all $j$. 
\end{example}

Notice that, in this case, $k$ is the largest eigenvalue of $P$ and
${\bf 1}$ is a left eigenvector of $P^r$, for all $r$. Therefore,
Expression (\ref{eq:cr}) can be written as
\begin{equation} \label{eq:cr1}
\frac{V.G^*.P^r.W^t}{V.G^*.P^r.{\bf 1}^t}= \frac{V.\phi.{\bf 1}.P^r.W^t}{V.\phi.{\bf 1}.P^r.{\bf 1}^t} = \frac{k^r  V.\phi.{\bf 1}.W^t}{k^r V.\phi.{\bf 1}.{\bf 1}^t}
\end{equation}
which is constant for all $r$.

\begin{example}{\bf $P$ has only one non-null eigenvalue and $Y$ is a Markov chain of finite order.}
Consider the Markov chain ${\bf X}$ with alphabet $A = \{1,2,3,4,5\}$
and transition matrix
  \begin{equation} \label{ex4}
P^X=\left(\begin{array}{ccccc}
0.3 & 0.2 & 0.1 & 0.25 & 0.15 \\
0.2 & 0.5 & 0.3 & 0.0 & 0.0 \\
0.85 & 0.0 & 0.0 & 0.1 & 0.05 \\
0.8 & 0.0 & 0.0 & 0.0 & 0.2 \\
1.0 & 0.0 & 0.0 & 0.0 & 0.0 
\end{array} \right).
 \end{equation}
\end{example}

Notice that, ${\bf Y}$ is not Markov of order 1 since the $W$ is not
constant. On the other hand, it is easy to see that $P$ has only one
non-null eigenvalue and its triangular form has 3 non-null
rows. Therefore, by Proposition \ref{prop:markov} we have an order 4
Markov chain.

In fact, it is easy to see that 
$$P^k/\lambda_1^k =  P^3/\lambda_1^3$$
for all $k \ge 3$. 

\subsection{Relation to Gibbsianess} 
\citeasnoun{fernandez/gallo/maillard/2011} proved that the chain ${\bf Y}$ is Gibbsian (in the sense of statistical physics) if and only if we have convergence, as $m$ and $n$ diverge, of 
\[
\P(Y_{0}=1|Y_{-m-1}=1,\,Y_{-m}^{-1}=0^m,\,Y_{1}^n=0^n,\,Y_{n+1}=1)=:p_{m,n}.
\]
This is because Gibbsianess (in the sense of statistical physics) corresponds to the continuity of the ``two-sided set of transition probabilities'', or \emph{specification}. We do not enter further in detailed definitions, and refer the interested reader to \citeasnoun{fernandez/gallo/maillard/2011}, where this notion of Gibbsianess is  defined.

  Let $r_{m,n}$ denote 
\[
\sum_{\tiny\begin{array}{ccc}b_{-m}^{-1}\in\pi^{-1}(0_{-m}^{-1})\\b_{1}^{n}\in\pi^{-1}(0_{1}^{n})\end{array}}\P(X_{-m}^{-1}=b_{-m}^{-1},\,X_{0}=1|X_{-m-1}=1)\P(X_{1}^{n}=b_{1}^{n},\,X_{n+1}=1|X_0=1)
\]
and $s_{m,n}$ denote
\[
\sum_{b_{-m}^{n}\in\pi^{-1}(0_{-m}^{n})}\P(X_{-m}^{n}=b_{-m}^{n},\,X_{n+1}=1|X_{-m-1}=1).
\]
Using similar calculations as the ones performed in the proof of Proposition \ref{prop:preliminary}, it is easy to arrive at
\[
p_{m,n}=\frac{r_{m,n}}{r_{m,n}+s_{m,n}}.
\]
This quantity also has a simple matrix form, which is as follows, $p_{0,0}=\frac{P_{1,1}^2}{P_{1,0}P_{0,1}+P_{1,1}^2}$, and
\[
p_{m+1,n+1}=\frac{VP^mW^t\,\,VP^nW^t}{VP^{m+n+1}{\bf 1}^t}, \,\,\,\,\mbox{ if } m,n\geq0 .
\]
The conditions for convergence, or not, of this quantity, and consequently, for Gibbsianess or non-Gibbsianess, are the same as for continuity. In particular, this means that, in the conditions we were considering in this paper, none of the image processes are at the same time non-Gibbsian and continuous, or \textrm{vice-versa}. 

In relation to the paper of \citeasnoun{chazottes/ugalde/2003}, it is important to notice that their notion of Gibbsianess is \emph{not} that of statistical physics. They consider Gibbs measures in the sense of Bowen, which is slightly different. This makes complicated a direct comparison between our results and theirs. 

To conclude, let us mention   other recent papers with interests related to ours. As already mentioned, \citeasnoun{verb:2011} and
\citeasnoun{verb:2011b} prove that, if the original measure has continuous transition probabilities (that is, not necessarily Markov measure, but it could be) with summable continuity rate (that is $\sum\beta_k<\infty$), then, factors of this measure still have continuous transition probabilities. Similar results were obtained concerning conservation of Bowen's Gibbsianess \cite{chazottes/ugalde/2011,pollicott/kempton/2011}, and of statistical physics Gibbsianess \cite{redig/wang/2010}.


\bibliographystyle{plain}
\bibliography{myref}
\end{document}